\documentclass[a4paper]{amsart}

\usepackage[
		pdftitle={Basic Analysis of Bin-Packing Heuristics},
		pdfauthor={Bastian Rieck}]{hyperref}

\usepackage{amsmath}
\usepackage{amsthm}
\usepackage{ebgaramond}
\usepackage{ebgaramond-maths}
\usepackage[scale=0.75]{inconsolata}
\usepackage{listings}
\usepackage{algorithmic}
\usepackage{algorithm}
\usepackage{longtable}
\usepackage{booktabs}
\usepackage{multirow}
\usepackage{graphicx}
\usepackage{rotating}

\usepackage[defaultlines=5, all]{nowidow}

\newtheorem{claim}{Claim}[section]

\title{Basic Analysis of Bin-Packing Heuristics}
\author{Bastian Rieck}
\urladdr{\url{https://bastian.rieck.me}}
\email{\href{mailto:bastian@rieck.me}{bastian@rieck.me}}

\begin{document}

\begin{abstract}
  The bin-packing problem continues to remain relevant in numerous
  application areas. This technical report discusses the empirical
  performance of different bin-packing heuristics for certain test 
  problems.
\end{abstract}

\maketitle

\section{Introduction} %

We assume that we are given a finite set of items
$I := \{i_1, i_2, \dots\}$ of different sizes $s(i_j) \in \mathbb{N}_{>
0}$, as well as a set of $K \in \mathbb{N}_{> 0}$ many bins of size $B$.
The bin-packing problem now refers to the question of whether there is
a partition of~$I$ into~$K$ disjoint sets such that the sum of all item
sizes in each of the sets is less than or equal to~$B$. Of particular
interest are solutions in which $K$ is being \emph{minimised}.
The bin-packing problem is highly relevant in numerous application
domains and is often studied with additional constraints~\cite{Balogh08, Balogh14}

As the bin-packing problem is known to be NP-hard~\cite{Garey90},
the development and subsequent analysis of appropriate heuristics are of
the utmost importance. This note provides an elementary analysis of
several common heuristics, while discussing their worst-case running time and
outlining techniques for improving their implementation.

\section{Caveat lector} %

This note is an extended version of a note the author wrote during his
undergraduate studies.\footnote{
  The note will remain available under \url{https://bastian.rieck.me/research/Note_BP.pdf}.
}
While the author is \emph{not} an expert in bin-packing algorithms or
heuristics, he sincerely hopes that this note and its accompanying
code\footnote{
  The code is available under \url{https://github.com/Pseudomanifold/bin-packing}.
} will prove at least as useful as the preceding note, which appears to
be cited whenever empirical results for bin-packing heuristics are
required. The reader should be aware that the heuristics chosen in this
note reflect the author's knowledge of the field more than 10 years ago.
This list is therefore not to be considered an authoritative source of
the field at present.

\section{Pseudocode \& analysis of worst-case running time} %

\subsection{First-Fit}

In pseudo-code, we have the following algorithm:

\begin{algorithm}[H]
\caption{First-Fit}
	\begin{algorithmic}[1]
		\FOR{All objects $i = 1,2,\dots,n$}
			\FOR{All bins $j = 1,2,\dots$}
				\IF{Object $i$ fits in bin $j$}
					\STATE Pack object $i$ in bin
					$j$.
					\STATE Break the loop and pack
					the next object.
				\ENDIF	
			\ENDFOR	
			\IF{Object $i$ did not fit in any available bin}
				\STATE Create new bin and pack object
				$i$.
			\ENDIF
		\ENDFOR
	\end{algorithmic}
\end{algorithm}

In the worst-case, a new bin has to be opened each time a new object is
inserted. Thus, there are $1,2,3,\dots,n-1$ executions of the inner
loop, which yields an asympotical factor of $\mathcal{O}(n^2)$.

\subsection{First-Fit-Decreasing} %

Since all weights of the objects are known \emph{prior} to running any
algorithm, \texttt{Counting Sort} is the best choice for sorting the
objects. 

\begin{algorithm}[H]
\caption{First-Fit-Decreasing}
	\begin{algorithmic}[1]
		\STATE Sort objects in decreasing order using \texttt{Counting Sort}.
		\STATE Apply \texttt{First-Fit} to the sorted list of
		objects.
	\end{algorithmic}
\end{algorithm}

Since \texttt{Counting Sort} has a complexity of $\mathcal{O}(n+k)$,
where k is the largest weight, the algorithm is obviously dominated by
the running time of \texttt{First-Fit}, which yields a factor of
$\mathcal{O}(n^2)$.

\subsection{Max-Rest} %

In pseudo-code, the following algorithm is used:

\begin{algorithm}[H]
\caption{Max-Rest}
	\begin{algorithmic}[1]
		\FOR{All objects $i = 1,2,\dots,n$}
			\STATE Determine $k = min\{i \mid c_i =
			\min_{j=1}^{j=m}c_j\}$, the index of the bin
			with maximum remaining capacity.
			\IF{Object $i$ fits in bin $k$}
				\STATE Pack object $i$ in bin $k$. 
			\ELSE
				\STATE Create new bin and pack object $i$.
			\ENDIF
		\ENDFOR
	\end{algorithmic}
\end{algorithm}

If a naive algorithm is used, determining the bin with maximum remaining
capacity yields a factor of $\mathcal{O}(n)$. Thus, the worst-case
running time of the algorithm is $\mathcal{O}(n^2)$. 

A more detailed analysis shows that the bin can be determined by using a
\emph{priority queue} (i.e.\ a heap). In this case, the bin can be
determined in constant time. Packing the object (either in a new bin or
in an existing one) then requires adding or updating an element of the
heap, which can be done in $\mathcal{O}(\log{}n)$. Hence, the improved
version of the algorithm has a worst-case running time of
$\mathcal{O}(n\log{}n)$:

\begin{algorithm}[H]
\caption{Max-Rest-Priority-Queue}
	\begin{algorithmic}[1]
		\FOR{All objects $i = 1,2,\dots,n$}
			\IF{Object $i$ fits in top-most bin of the
			priority queue}
				\STATE Remove top-most bin from queue.
				\STATE Add object $i$ to this bin.
				\STATE Push updated bin to queue.
			\ELSE
				\STATE Create new bin and pack object $i$.
			\ENDIF
		\ENDFOR
	\end{algorithmic}
\end{algorithm}

\subsection{Next-Fit} %

In pseudo-code, my implementation proceeds as follows:

\begin{algorithm}[H]
\caption{Next-Fit}
	\begin{algorithmic}[1]
		\FOR{All objects $i = 1,2,\dots,n$}
			\IF{Object $i$ fits in current bin}
				\STATE Pack object $i$ in current bin. 
			\ELSE
				\STATE Create new bin, make it the
				current bin, and pack object $i$.
			\ENDIF
		\ENDFOR
	\end{algorithmic}
\end{algorithm}

Since packing an object can be done in constant time, the algorithm is
dominated by the loop, which has a running time of $\Theta(n)$.

\subsection{Next-Fit-Decreasing} %

The algorithm is straightforward:

\begin{algorithm}[H]
\caption{Next-Fit-Decreasing}
	\begin{algorithmic}[1]
		\STATE Sort objects in decreasing order using
		\texttt{Counting Sort}.
		\STATE Apply \texttt{Next-Fit} to the sorted list of
		objects.
	\end{algorithmic}
\end{algorithm}

Since \texttt{Next-Fit} has a running time of $\Theta(n)$, the
dominating factor is the \texttt{Counting Sort} algorithm, which has a
running time of $\mathcal{O}(n+k)$, where $k$ is the maximum weight of
the problem.

\subsection{Best-Fit} %

The algorithm works like this:

\begin{algorithm}[H]
\caption{Best-Fit}
	\begin{algorithmic}[1]
		\FOR{All objects $i = 1,2,\dots,n$}
			\FOR{All bins $j = 1,2,\dots$}
				\IF{Object $i$ fits in bin $j$}
					\STATE Calulate remaining
					capacity after the object has
					been added.
				\ENDIF
			\ENDFOR	

			\STATE Pack object $i$ in bin $j$, where $j$ is
			the bin with minimum remaining capacity after
			adding the object (i.e. the object ``fits
			best'').
			\STATE If no such bin exists, open a new one and
			add the object.
		\ENDFOR
	\end{algorithmic}
\end{algorithm}

Since \emph{all bins} are examined in each step, the algorithm has an
obvious running time of $\mathcal{O}(n^2)$. The running time can be
decreased by using a heap in order to determine the best bin:

\begin{algorithm}[H]
\caption{Best-Fit-Heap}
	\begin{algorithmic}[1]
		\FOR{All objects $i = 1,2,\dots,n$}
			\STATE Perform a Breadth-First-Search in the
			heap to determine the best bin.
			\IF{Best bin has been found}
				\STATE Pack object $i$ in this bin.
				\STATE Restore the heap property.
			\ELSE
				\STATE Open a new bin and add the
				object.
			\ENDIF
		\ENDFOR
	\end{algorithmic}
\end{algorithm}

Using a heap decreases the running time slightly. An even more rapid
implementation of this algorithm will be outlined later on.

\section{Results for the test problems} %

The benchmarks have been performed using an Intel Celeron M $1.5$ GHz.
The results are not too surprising: Obviously, the \texttt{Next-Fit}
heuristic is fastest because only 1 bin has to be managed. However, due
to the efficient data structure (a priority queue) that has been used
for the \texttt{Max-Rest} heuristic, this heuristic will generally be
almost as fast as \texttt{Next-Fit}. Furthermore, the
implementation of the \texttt{Best-Fit} heuristic has a
worst-case running time of $\mathcal{O}(Kn)$, where $K$ is the
maximum weight. Thus, the slowest algorithms are \texttt{First-Fit} and
\text{First-Fit-Decreasing}.

Detailed results can be studied in the table below. In each set, the
best solutions are marked using the dagger symbol ``$\dagger$''. The
timing is not accurate for the small running times. This is due to the
\texttt{CLOCKS\_PER\_SEC} macro that has been used for the benchmarks.
Hence, in some cases, the same running times will appear. This means
that the running times differ by a very small amount (measured in
``raw'' CPU cycles). A value of $0$ signifies that the measurement is
outside the notable range.

The algorithms have been compiled with the \texttt{-O3} optimizations of
the \texttt{g++} compiler. All heuristics have been abbreviated to fit
in the table. Thus, \texttt{MR} is \texttt{Max-Rest}, for example. The
$+$-signs after the algorithm names specify whether an optimized version
of the algorithm has been used. For implementation details, refer to
Table~\ref{tab:Naming_Scheme} on p.~\pageref{tab:Naming_Scheme}.

\begin{center}
	\begin{longtable}{
				p{0.25\textwidth}
				p{0.25\textwidth}
				p{0.25\textwidth}
				p{0.25\textwidth}
			 }
		\caption{Results for the test problems}\\
		\toprule	 
		Problem set & Algorithm & Bins & Time in s\\
		\midrule
		\endhead
		\multirow{6}{*}{\tt bp1.dat}	& \tt MR+ 	& $628$ 		& $0$\\
						& \tt FF++	& $564$ 		    & $0$\\
						& \tt FFD++	& $545$ 		    & $0$\\
						& \tt NF	  & $711$ 		    & $0$\\
						& \tt NFD+	& $686$ 		    & $0$\\
						& \tt BF++	& $553^\dagger$	& $0$\\
		\midrule
		\multirow{6}{*}{\tt bp2.dat}	& \tt MR+ 	& $6131$ 			& $0^\dagger$\\
						& \tt FF++	& $5420$ 			& $0^\dagger$\\
						& \tt FFD++	& $5321^\dagger$ 	& $0^\dagger$\\
						& \tt NF	  & $6986$ 			& $0^\dagger$\\
						& \tt NFD+	& $6719$ 			& $0^\dagger$\\
						& \tt BF++	& $5377$			& $0^\dagger$\\
		\midrule
		\multirow{6}{*}{\tt bp3.dat}	& \tt MR+ 	& $16637$ 		& $0^\dagger$\\
						& \tt FF++	& $16637$ 		& $0.0071825$\\
						& \tt FFD++	& $10000^\dagger$	& $0^\dagger$\\
						& \tt NF	  & $16637$ 		& $0^\dagger$\\
						& \tt NFD+	& $16637$ 		& $0^\dagger$\\
						& \tt BF++	& $16637$			& $0^\dagger$\\
		\midrule
		\multirow{6}{*}{\tt bp4.dat}	& \tt MR+ 	& $29258$ & $0.0078125$\\
						& \tt FF++	& $25454$ 		& $0.015625$\\
						& \tt FFD++	& $25157^\dagger$	& $0.0078125$\\
						& \tt NF	  & $33397$ 		& $0.0078125$\\
						& \tt NFD+	& $32174$ 		& $0^\dagger$\\
						& \tt BF++	& $25303$			& $0^\dagger$\\
		\midrule
		\multirow{6}{*}{\tt bp5.dat}	& \tt MR+ 	& $32524$ 		& $0^\dagger$\\
						& \tt FF++	& $30155$ 		& $0.0078125$\\
						& \tt FFD++	& $30111^\dagger$	& $0^\dagger$\\
						& \tt NF	  & $36623$ 		& $0.0078125$\\
						& \tt NFD+	& $37048$ 		& $0^\dagger$\\
						& \tt BF++	& $30152$ 		& $0^\dagger$\\
		\midrule
		\multirow{6}{*}{\tt bp6.dat}	& \tt MR+ 	& $55566$ & $0.0234375$\\
						& \tt FF++	& $50021$ 		& $0.0078125$\\
						& \tt FFD++	& $49951^\dagger$	& $0.0078125$\\
						& \tt NF	  & $63078$ 		& $0.0078125$\\
						& \tt NFD+	& $59852$ 		& $0^\dagger$\\
						& \tt BF++	& $50021$			& $0.0078125$\\
		\midrule
		\multirow{6}{*}{\tt bp7.dat}	& \tt MR+ 	& $38242$ 		& $0.015625$\\
						& \tt FF++	& $36863^\dagger$	& 0.0078125\\ 
						& \tt FFD++	& $39276$ 		& $0.0071825$\\
						& \tt NF	  & $42082$ 		& $0^\dagger$\\
						& \tt NFD+	& $47937$ 		& $0^\dagger$\\
						& \tt BF++	& $36866$			& $0.0703125$\\
		\midrule
		\multirow{6}{*}{\tt bp8.dat}	& \tt MR+ 	& $82804$ & $0.0234375$\\
						& \tt FF++	& $79746$ 		& $0.015625$\\
						& \tt FFD++	& $79130^\dagger$	& $0.015625$\\
						& \tt NF	  & $93618$ 		& $0.0078125^\dagger$\\
						& \tt NFD+	& $104096$ 		& $0.0078125^\dagger$\\
						& \tt BF++	& $79731$			& $0.015625$\\
		\midrule
		\multirow{6}{*}{\tt bp9.dat}	& \tt MR+ 	& $293319$ 		& $0.078125^\dagger$\\
						& \tt FF++	& $252577$ 		& $0.234375$\\
						& \tt FFD++	& $251164^\dagger$	& $0.125$\\
						& \tt NF	  & $333852$ 		& $0.0078125^\dagger$\\
						& \tt NFD+	& $322643$ 		& $0.015625$\\
						& \tt BF++	& $251568$		& $0.0390625$\\
		\midrule
		\multirow{6}{*}{\tt bp10.dat}	& \tt MR+	& $588478$ 		& $0.171875$\\
						& \tt FF++	& $506844$ 		      & $0.296875$\\
						& \tt FFD++	& $504927^\dagger$	& $0.179688$\\
						& \tt NF	  & $669926$ 		      & $0.015625^\dagger$\\
						& \tt NFD+	& $645906$ 		      & $0.0390625$\\
						& \tt BF++	& $505643$		      & $0.046875$\\
		\midrule
		\multirow{6}{*}{\tt bp11.dat}	& \tt MR+ 	& $2929609$ & $0.882812$\\
						& \tt FF++	& $2510868$ 		      & $8.42188$\\
						& \tt FFD++	& $2502387^\dagger$	  & $4.86719$\\
						& \tt NF	  & $3333928$ 		& $0.09375^\dagger$\\
						& \tt NFD+	& $3225064$ 		& $0.234375$\\
						& \tt BF++	& $2504284$		  & $1.3125$\\
		\bottomrule
	\end{longtable}
\end{center}

\section{Speed-up for choosing the right bin} %

Algorithms \texttt{First-Fit-Decreasing} and
\texttt{Next-Fit-Decreasing} greatly benefit from changing the sorting
algorithm to \texttt{Counting Sort}, thus yielding a factor of
$\mathcal{O}(n)$ for sorting instead of the usual
$\mathcal{O}(n\log{}n)$ for Heapsort.

\emph{All algorithms} benefit from the following idea: Since the minimum
weight $w$ is known, a bin can be closed after adding an object if the
remaining capacity of the bin is less than the ``limit capacity'' $c_l =
K-w$ (where $K$ is the total capacity of the bin).

If not mentioned otherwise, optimized versions of all algorithms have
been implemented. Figures \ref{img:Comparison_5} and
\ref{img:Comparison_8} show the running times of all algorithms for
problems \texttt{bp5} and \texttt{bp8}, respectively. The algorithms
depicted in the figures are explained in Table~\ref{tab:Naming_Scheme}.

\begin{table}[b]
	\centering
	\caption{Naming schemes used in figures \ref{img:Comparison_5} and
	\ref{img:Comparison_8}}
	\begin{tabular}{lp{0.9\textwidth}}
		\toprule
		\textbf{MR}	& \texttt{Max-Rest}\\
		\textbf{MR+}	& \texttt{Max-Rest} using a \emph{priority queue}\\
		\midrule
		\textbf{FF}	& \texttt{First-Fit}\\
		\textbf{FF+}	& \texttt{First-Fit} using the \texttt{vector}
				  container from the STL. Thus, ``almost full'' bins can be
				  removed without having to re-order
				  the array.\\
		\textbf{FF++}	& \texttt{First-Fit} using the \texttt{map}
				  container from the STL. Thus, a lookup table is created that
				  determines the proper bin more rapidly.\\
		\textbf{FFD}	& \texttt{First-Fit-Decreasing} using the \texttt{Heapsort} algorithm.\\
		\textbf{FFD+}	& \texttt{First-Fit-Decreasing} using
				  \textbf{FF+} and \texttt{Counting Sort}.\\ 
		\textbf{FFD++}	& \texttt{First-Fit-Decreasing} using
				  \textbf{FF++} and \texttt{Counting Sort}.\\
		\midrule
		\textbf{NF}	& \texttt{Next-Fit}\\
		\textbf{NFD}	& \texttt{Next-Fit-Decreasing}\\
		\textbf{NFD+}	& \texttt{Next-Fit-Decreasing} using
				  \texttt{Counting Sort}.\\
		\midrule
		\textbf{BF}	& \texttt{Best-Fit}\\
		\bottomrule
	\end{tabular}
	\label{tab:Naming_Scheme}
\end{table}

\begin{figure}[bp]
  \centering
  \includegraphics[width=\linewidth]{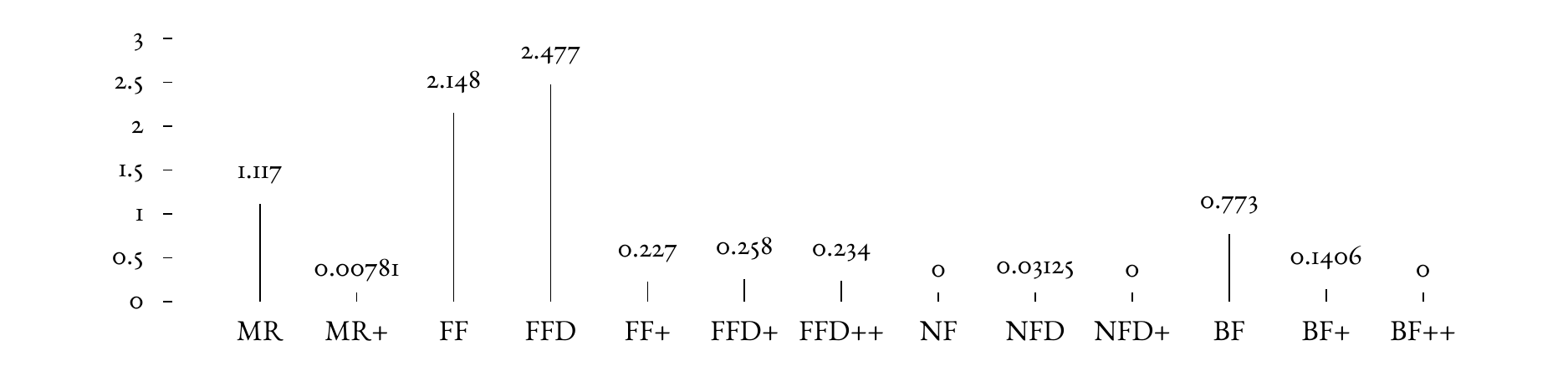}
	\caption{Running times in seconds for problem \texttt{bp5}. Note
	that a running time of $0$ simply means that only very few CPU
	cycles have been used (in this case, the \texttt{clock()}
	function will not be able to take accurate measurements).}
	\label{img:Comparison_5}
\end{figure}

\begin{sidewaysfigure}
  \centering
  \includegraphics[width=0.75\linewidth]{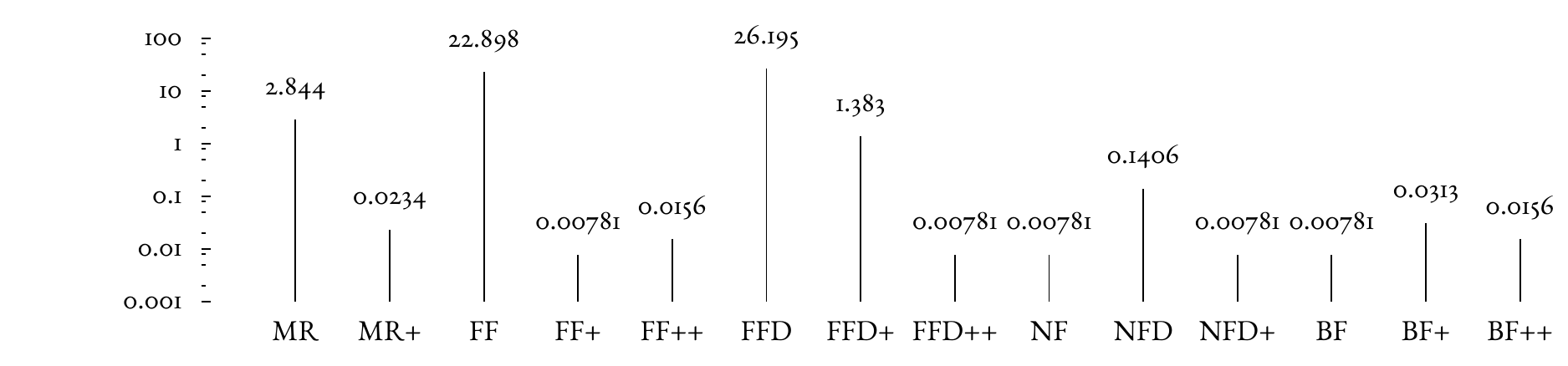}
  \caption{Running times for problem \texttt{bp8}. Note that a
	logarithmic scale has been used because the running times differ
	greatly.}
	\label{img:Comparison_8}
\end{sidewaysfigure}

\subsection{First-Fit \& First-Fit-Decreasing} %

The right bin can be determined very quickly if the algorithm uses a
\emph{map} that stores the index $j$ of the bin which contains an
object of weight $w$. Since the bins are filled in increasing order of
their indices, all bins with index $j\prime < j$ can be skipped when
adding an object of weight $w\prime \geq w$. Asymptotically, the
algorithm is still in $\mathcal{O}(n^2)$. 

A sample implementation has been supplied with the source code. See
function \texttt{first\_fit\_map} for more details. In pseudo-code, we
have:

\begin{algorithm}[H]
\caption{First-Fit-Lookup}
	\begin{algorithmic}[1]
		\FOR{All objects $i = 1,2,\dots,n$}
			\STATE Lookup the last bin $j$ that was used to
			pack an object of this size. If no such bin
			exists, set $j = 1$.
			\FOR{All remaining bins $j,j+1,\dots$} 
				\IF{Object $i$ fits in current bin}
					\STATE Pack object $i$ in
					current bin.
					\STATE Save the current index in
					the lookup table.
					\STATE Break the loop and pack
					the next object.
				\ENDIF
			\ENDFOR	
			\IF{Object $i$ did not fit in any available bin}
				\STATE Create new bin and pack object
				$i$.
			\ENDIF
		\ENDFOR
	\end{algorithmic}
\end{algorithm}

\subsection{Max-Rest} %

This algorithm greatly benefits from a \emph{priority queue} (see the running
time analysis): When the heuristic needs to decide where to place an
object, the bin with the highest priority (i.e. the \emph{maximum} remaining
capacity) would be chosen from the queue. If the object does not fit, a
new bin needs to be created. If it does fit, the bin's capacity is
updated and it is placed back in the priority queue.

A sample implementation has been supplied with the source code. See
function \texttt{max\_rest\_pq} for more details.

\subsection{Best-Fit} %

The \texttt{Best-Fit} algorithm greatly benefits from using a
\emph{lookup table} for the bin sizes. Index $i$ of this table stores
the number of bins with remaining capacity $i$. When a new object is
added, the lookup table is queried for increasing remaining capacities.
Thus, the first match will be the best one. The algorithm can be
formalized as follows: 

\begin{algorithm}[H]
\caption{Best-Fit-Lookup-Table}
	\begin{algorithmic}[1]
		\STATE Initialize lookup table $t$ with $t_K = n$ (there
		are $n$ bins with remaining capacity $K$, hence no
		object has been packed yet). 
		\FOR{All objects $i = 1,2,\dots,n$}
			\STATE Let $w_i$ be the current object size.
			\STATE Search a suitable bin by querying the
			lookup table for bins of remaining capacity
			$w_i, w_i+1, \dots$ until a bin has been found
			at index $t_l$.
			\STATE $t_l--$, because there is one bin less
			with remaining capacity $l$.
			\STATE $t_{l-w_i}++$, because the number of bins
			with remaining capacity $l-w_i$ has been
			increased by one. 
		\ENDFOR
		\STATE Compute $\sum_{i=0}^{i=K-1} t_i$ in order to
		determine the number of bins that has to be used.
	\end{algorithmic}
\end{algorithm}

Since the remaining capacity of a bin is at most $K$, a suitable bin can
always be determined in $\mathcal{O}(K)$. Because of the outer loop, the
algorithm has a running time of $\mathcal{O}(nK)$.

The obvious disadvantage of this algorithm is that the object positions
are not stored in the lookup table. This can be solved by using
a table of queues, for example, which contain pointers to the ``real''
bins. In the last two steps of the outer loop, the elements would be
removed or added to the queues, respectively. Since this can be done in
constant time, the running time of the algorithm is \emph{not} changed.

\subsection{Next-Fit and Next-Fit-Decreasing} %

Since only one bin is managed at a time, these algorithms cannot be
optimized any further (apart from using \texttt{Counting Sort} for
\texttt{Next-Fit-Decreasing}).

\section{Worst-case results} %

\begin{claim}
	$m_{\mathrm{FF}} < \frac{17}{10}\cdot{}m_{\mathrm{OPT}}+2$
\end{claim}

\begin{proof}
	Already shown in lecture.
\end{proof}

\begin{claim}
	$m_{\mathrm{FFD}} < \frac{11}{9}\cdot{}m_{\mathrm{OPT}}+4$
\end{claim}

\begin{proof}
	Already shown in lecture.
\end{proof}

\begin{claim}
	$m_{\mathrm{NF}} \leq 2\cdot{}m_{\mathrm{OPT}}$
\end{claim}

\begin{proof}
	For two subsequent bins, $c_i+c_{i+1} > 1$ holds. If it were
	otherwise, no new bin would have been opened. Hence, the
	solution of the heuristic is at most twice as big as the optimum
	solution. The bound is \emph{tight} for a problem that contains
	$n$ objects with weights $\frac{1}{2}$ and $n$ objects with
	weights $\frac{1}{2n}$. If the objects arrive in an alternating
	fashion, i.e. $\frac{1}{2}, \frac{1}{2n}, \frac{1}{2}, \dots$,
	\texttt{Next-Fit} will create $2n$ bins, whereas the optimal
	solution consists of $n+1$ bins. 
\end{proof}

\bibliographystyle{amsalpha}
\bibliography{main}

\vfill

\end{document}